\documentclass[11pt]{amsart}
\usepackage[top = 1in, bottom = 1in, left = 1.5in, right=1.5in,
marginparwidth=1.3in]{geometry}
\usepackage{setspace}
\usepackage{amsmath, amsfonts, amsthm, amstext, xspace}
\usepackage[mathscr]{euscript}
\usepackage{amssymb,latexsym}
\usepackage{comment}
\usepackage{xcolor}
\usepackage{hyperref}
\usepackage{xcolor}
\hypersetup{
    colorlinks,
    linkcolor={red!50!black},
    citecolor={blue!50!black},
    urlcolor={blue!80!black}
}
\usepackage[capitalise, nameinlink]{cleveref}
\crefname{subsection}{Subsection}{Subsections}
\usepackage{adjustbox}
\definecolor{seagreen}{RGB}{46,139,87}
\definecolor{maroon}{RGB}{128,0,0}
\definecolor{darkviolet}{RGB}{148,0,211}
\definecolor{twelve}{RGB}{100,100,170}
\definecolor{thirteen}{RGB}{100,150,50}
\definecolor{fourteen}{RGB}{200,0,0}
\definecolor{fifteen}{RGB}{0,200,0}
\definecolor{sixteen}{RGB}{0,0,200}
\definecolor{seventeen}{RGB}{200,0,200}
\definecolor{eighteen}{RGB}{0,200,200}

\usepackage{centernot}
\usepackage{longtable}
\usepackage{mathtools}
\usepackage{stmaryrd}
\makeatletter
\newcommand{\xMapsto}[2][]{\ext@arrow 0599{\Mapstofill@}{#1}{#2}}
\def\Mapstofill@{\arrowfill@{\Mapstochar\Relbar}\Relbar\Rightarrow}
\makeatother

\usepackage[all]{xy}

\theoremstyle{definition} 

\newtheorem{thm}{Theorem}[section]
\newtheorem*{theorem*}{Theorem}
\newtheorem*{conjecture*}{Conjecture}
\newtheorem*{corollary*}{Corollary}
\newtheorem{lemma}[thm]{Lemma}
\newtheorem{theorem}[thm]{Theorem}

\newtheorem{proposition}[thm]{Proposition}

\theoremstyle{definition}

\newtheorem{remark}[thm]{Remark}

\newtheorem{thmx}{Theorem}

\numberwithin{equation}{section} 

\def\a{\mathbb{A}}

\def\c{\mathbb{C}}

\def\e{\mathbb{E}}
\def\f{\mathbb{F}}
\def\bbF{\mathbb{F}}
\def\g{\mathbb{G}}

\def\p{\mathbb{P}}
\def\q{\mathbb{Q}}

\def\r{\mathbb{R}}

\def\z{\mathbb{Z}}
\def\bbZ{\mathbb{Z}}

\def\Mod{\operatorname{Mod}}

\def\Spec{\operatorname{Spec}}

\def\Ext{\operatorname{Ext}}

\def\colim{\operatorname{colim}}

\def\fib{\operatorname{fib}}

\def\SH{\operatorname{SH}}

\def\Sp{\operatorname{Sp}}

\def\Cell{\operatorname{Cell}}
\def\cell{\text{cell}}
\def\cof{\operatorname{cof}}

\usepackage{marginnote}
\usepackage{soul}

\newcommand{\attop}[1]{{\let\textstyle\scriptstyle\let\scriptstyle\scriptscriptstyle\substack{#1}}}
\renewcommand{\atop}[1]{{\let\scriptstyle\textstyle\let\scriptscriptstyle\scriptstyle\substack{#1}}}

\newcommand{\map}{\operatorname{map}}

\newcommand{\bone}{\mathbf{1}}

\newcommand{\op}{\text{op}}
\newcommand{\Fun}{\operatorname{Fun}}
\newcommand{\Be}{\operatorname{Be}}

\definecolor{llteal}{RGB}{198,232,227}
\definecolor{llred}{RGB}{237,228,228}
\definecolor{llgray}{RGB}{230,230,230}
\definecolor{maroon}{RGB}{150,0,0}
\definecolor{orange}{RGB}{255,165,0}

\newcommand{\highlight}[1]{\ifmmode{\text{\sethlcolor{llgray}\hl{$#1$}}}\else{\sethlcolor{llred}\hl{#1}}\fi}

\usepackage{amssymb}
\usepackage{tikz-cd}

\author{William Balderrama}\address{University of Virginia}\email{eqr8nm@virginia.edu}
\author{Kyle Ormsby}\address{Reed College / University of Washington}\email{ormsbyk@reed.edu \textnormal{/} ormsbyk@uw.edu}
\author{J.D. Quigley}\address{University of Oregon}\email{jquigley1993@gmail.com}

\title[A motivic analogue of the $K(1)$-local sphere spectrum]{A motivic analogue of \\ the $K(1)$-local sphere spectrum}

\begin{document}

\begin{abstract}
We identify the motivic $KGL/2$-local sphere as the fiber of $\psi^3-1$ on $(2,\eta)$-completed Hermitian $K$-theory, over any base scheme containing $1/2$. This is a motivic analogue of the classical resolution of the $K(1)$-local sphere, and extends to a description of the $KGL/2$-localization of an arbitrary motivic spectrum. Our proof relies on a novel conservativity argument that should be of broad utility in stable motivic homotopy theory.
\end{abstract}

\maketitle
\tableofcontents

\section{Introduction}

The motivic stable homotopy category $\SH(k)$ over a field $k$ was introduced by Morel and Voevodsky \cite{MV99} in the 1990's to apply powerful tools from algebraic topology to problems in algebraic geometry. Its successes include the resolution of the Milnor and Bloch--Kato conjectures \cite{Voe03,Voev03,OVV,Voe11,hasemeyerweibel}, new computations in algebraic and Hermitian $K$-theory \cite{holimHerm,slicesherm}, detailed analyses of algebraic cobordism \cite{motivicBP}, contributions to the developing field of quadratic enumerative geometry \cite{kwEKL,levinequadenum}, and computational input to the Asok--Fasel--Hopkins program on algebraic vector bundles \cite{af1,af2,afh}.

There has been considerable interest in studying localizations of the motivic stable homotopy category. Morel showed in \cite{morelrational} that $\SH(k)[\frac{1}{2}]$ splits as a product of its ``positive" and ``negative" parts, and Cisinski--D{\'e}glise \cite{CD19}, Garkusha \cite{garkusha_2019}, and D{\'e}glise--Fasel--Khan--Jin \cite{DFJK21} have analyzed the rationalization $\SH(k)_\q$. In particular, the positive part is related to the theory of rational motives, while the negative part is related to Witt theory and $\eta$-periodic phenomenon.

Here $\eta \in \pi_{1,1}(\bone_k)$ is represented by the canonical $\g_m$-torsor $\a^2\smallsetminus 0\to \p^1$ and may be viewed as a graded endomorphism of the unit object $\bone_k$ of $\SH(k)$. The coefficients of the $\bone_k[\eta^{-1}]$ have been studied over various base fields (\cite{AM17, CQ21, GI15, GI16, OR20, Wil18}), and Bachmann--Hopkins \cite[Thm. 1.1]{BH20} recently showed that
$$\bone_k[\eta^{-1}]_{(2)} \simeq \fib\left( \varphi \colon kw_{(2)} \to \Sigma^4 kw_{(2)} \right).$$
Here, $kw = KQ[\eta^{-1}]_{\geq 0}$ is the spectrum of connective Balmer--Witt K-theory and $\varphi$ is a lift of $\psi^3-1$. 

Chromatic localizations play a prominent role in the homotopy theory of topological spectra. Readers unfamiliar with Morava $E$- and $K$-theories and their roles in stratifying the stable homotopy category via the moduli stack of one-dimensional commutative formal group laws may consult \cite{ravenelorange}. Many of these constructions have been ported to the motivic context \cite{Bor09,Hor18,hoprimes,joachimi,Kra18}, and their role in determining thick subcategories, nilpotence, and periodicity in stable motivic homotopy theory remains a topic of contemporary interest.

In this work, we study a motivic analogue of $K(1)$-localization at the prime $2$. (The assumption $p=2$ is implicit henceforth.) In the topological setting, $K(1) = KU/2$, and localization at $K(1)$ corresponds to the first monochromatic layer. Famously, there is a short ``resolution'' of the $K(1)$-local sphere in terms of real $K$-theory $KO$ (see \cite{adams1974operations}, \cite[Theorem 4.3]{Bou79}, \cite[Lemma 8.7]{Rav84}, and \cite[Lemma 2.3]{hopkinsmahowaldsadofsky1995constructions}). Define the spectrum $J$ to be the fiber of $\psi^3-1$ acting on $2$-complete real $K$-theory
\[
  J\longrightarrow KO^\wedge_2\xrightarrow{\psi^3-1}KO^\wedge_2;
\]
then $L_{K(1)}X = (X\otimes J)^\wedge_2$ for any spectrum $X$. (Here $\psi^3$ is the third Adams operation on $KO^\wedge_2$; it may be replaced with any $\psi^g$ for $g$ a topological generator of $\z_2^\times/\{\pm 1\}$.) This result provides tremendous computational control over $K(1)$-localizations. Indeed, the induced long exact sequence demonstrates that the groups $\pi_*L_{K(1)}X$ only depend on the $KO^\wedge_2$-homology of $X$, the action of $\psi^3$ thereon, and potential additive extensions.

We have used the traditional name $J$ for $L_{K(1)}S$ because of the tight relationship between this spectrum and classes in the image of the $J$-homomorphism. First constructed by Mahowald, the \emph{connective image-of-$J$ spectrum} $j$ is the fiber of a lift of $\psi^3-1$ to a map $ko\to \Sigma^4ksp$. (Here $ksp$ is connective quaternionic $K$-theory, and $\Sigma^4 ksp\simeq bspin$ is the $3$-connected cover of $KO$.) The spectrum $j$ admits a unit $\bone = S^0\to j$ that detects the image of the classical $J$-homorphism $\pi_{*+1}KO\to \pi_*S^0$ and the simple $2$-torsion Adams classes in degrees $8k+1$ and $8k+2$ \cite{LM87}. It is also the case that the natural map $j\to J$ induces an isomorphism on $\pi_n$ for $n\ge 2$.

Our motivic analogue of $J$ follows a similar pattern but replaces $KO$ with the Hermitian $K$-theory $\p^1$-spectrum $KQ$ representing the $K$-theory of symmetric bilinear forms. To be more precise, fix a base scheme $S$ containing $1/2$. The $S$-motivic spectra $KGL$ and $KQ$ of algebraic and Hermitian $K$-theory are motivic refinements of the classical spectra $KU$ and $KO$ of complex and real $K$-theory. (For convenience, we use the definitions of \cite[\S 3.2.5]{BH20} for $KGL$ and $KQ$ over $S$, and note that they represent algebraic and Hermitian $K$-theory, respectively, as long as $S$ is Noetherian and regular.) We define the motivic spectrum $JQ$ (the \emph{quadratic $J$-spectrum}) to be the fiber of $\psi^3-1$ acting on $KQ^\wedge_{(2,\eta)}$ so that there is a fiber sequence
\[
  JQ \longrightarrow KQ_{(2,\eta)}^\wedge\xrightarrow{\psi^3-1} KQ_{(2,\eta)}^\wedge.
\]

\begin{thmx}\label{MT:LX}
Let $X$ be a motivic spectrum over a base scheme $S$ containing $1/2$. Then there is a fiber sequence
\[
L_{KGL/2}X \to (X \otimes KQ)_{(2,\eta)}^\wedge \xrightarrow{\psi^3-1} (X \otimes KQ)_{(2,\eta)}^\wedge,
\]
or equivalently,
\[
  L_{KGL/2}X\simeq (X\otimes JQ)^\wedge_{(2,\eta)}.
\]
\end{thmx}

\begin{remark}
Belmont--Isaksen--Kong \cite{BIK22} and Kong and the third author \cite{KQ22} studied a ``connective" analogue of $JQ$ over a variety of base fields $k$, there denoted
$$L = \fib\left(\psi^3-1\colon kq_{(2,\eta)}^{\wedge} \rightarrow kq_{(2,\eta)}^\wedge\right).$$
Here, $kq = \tilde{f}_0 KQ$ is the very effective cover of Hermitian K-theory \cite{ARO17}. The very effective cover functor is not triangulated, so the computations in \cite{BIK22, KQ22} do not immediately yield the homotopy groups of $\tilde{f}_0 L_{KGL/2}\bone_k$; the precise relationship between these groups will be described in a forthcoming revision of \cite{KQ22}. 
\end{remark}

\begin{remark}
The $C_2$-equivariant Betti realization functor $\Be\colon \SH(\r)\rightarrow\Sp^{C_2}$ sends $KGL$ and $KQ$ to the $C_2$-equivariant spectra $K\r$ and $KO_{C_2}$ respectively. Thus the $\r$-motivic $KGL/2$-local sphere may be viewed as a lift of the $C_2$-equivariant $K\r/2$-local sphere $L_{K\r/2}S_{C_2}\simeq \fib\left(\psi^3-1\colon (KO_{C_2})_{2}^\wedge\rightarrow (KO_{C_2})_{2}^\wedge\right)$ analyzed by the first author in \cite{balderrama2021borel}.
\end{remark}

After fixing some notation in \cref{ssec:conventions}, we outline the proof of \cref{MT:LX} in \cref{ssec:outline} below.

\subsection{Conventions}\label{ssec:conventions}

We maintain the following conventions throughout the paper.

\begin{enumerate}
\item We work over a base scheme $S$ containing $1/2$.
\item We write $\SH(S)$ for the category of $S$-motivic spectra.
\item We write $\bone_S \in \SH(S)$ for the $S$-motivic sphere spectrum.
\item We abbreviate $\SH(\Spec(k)) = \SH(k)$, $\bone_k = \bone_{\Spec(k)}$, and so forth.
\item We write $KGL$ and $KQ$ for the $S$-motivic algebraic and Hermitian $K$-theory spectra respectively (see \cite[\S 3.2]{BH20}, where $KQ$ is denoted $KO$).
\item For a field $k$ containing $1/2$, we write $k^M_n(k) = K^{M}_n(k)/2$ for the $n$th mod $2$ Milnor $K$-groups of $k$.
\item We abbreviate $\rho := [-1] \in \pi_{-1,-1}\bone_S$ for the class represented by the unstable map $\Spec(S)_+\to \a^1\smallsetminus 0$ induced by the map $S[x]\to S$ evaluating at $-1$.
\item We abbreviate $JQ = \fib\left(\psi^3-1\colon KQ_{(2,\eta)}^\wedge \to KQ_{(2,\eta)}^\wedge\right)$, so that in general $(X\otimes JQ)_{(2,\eta)}^\wedge = \fib\left(\psi^3-1\colon (X \otimes KQ)_{(2,\eta)}^\wedge \to (X \otimes KQ)_{(2,\eta)}^\wedge\right)$.
\item We write $\map(X,Y)$ for the spectrum of maps between objects $X$ and $Y$.

\end{enumerate}

\subsection{Outline of proof}\label{ssec:outline}
We now outline the proof of \cref{MT:LX}. Using the above notation, the theorem asserts that if $X$ is an $S$-motivic spectrum then the natural map $X\rightarrow (X\otimes JQ)^\wedge_{(2,\eta)}$ realizes $(X\otimes JQ)^\wedge_{(2,\eta)}$ as the $KGL/2$-localization of $X$. By definition, to prove this we must verify that $(X\otimes JQ)_{(2,\eta)}^\wedge$ is $KGL/2$-local, and that $X\rightarrow (X\otimes JQ)_{(2,\eta)}^\wedge$ is a $KGL/2$-equivalence.

The fact that $(X\otimes JQ)_{(2,\eta)}^\wedge$ is $KGL/2$-local follows from a formal argument using the fact that $KGL\simeq KQ/\eta$. The same line of reasoning reduces checking that $X\rightarrow (X\otimes JQ)_{(2,\eta)}^\wedge$ is a $KGL/2$-equivalence for any $X$ to just the case $X = \bone_S$. We give these simple reductions in \cref{Sec:FirstReduction}. The bulk of our work is then to verify that the map $\bone_S\rightarrow JQ$ is a $KGL/2$-equivalence. The proof of this splits naturally into two parts: the reduction to $S = \Spec(\c)$, carried out in \cref{Sec:DtoC}, and the case $S = \Spec(\c)$, carried out in \cref{Sec:C}.

These proceed as follows. First, the assertion that $\bone_S\rightarrow JQ$ is a $KGL/2$-equivalence is stable under base change, and this allows us to reduce to considering just $S = \Spec(\bbZ[\tfrac{1}{2}])$. To then reduce to $S = \Spec(\c)$, we prove the following general conservativity result: base change  $\SH(\z[\tfrac{1}{2}])\rightarrow\SH(\c)$ is conservative when restricted to the full subcategory of $2$-torsion cellular motivic spectra with vanishing real Betti realization and convergent effective slice tower (\cref{thm:conservative}). Once we have reduced to $S = \Spec(\c)$, we make use of the close relation between $\c$-motivic homotopy theory and the classical Adams--Novikov spectral sequence to further reduce to the classical description of the non-motivic $K(1)$-local sphere.

\begin{remark}
Pelaez and Weibel show in \cite{pelaezweibel2014slices} that the $KGL$-cooperation algebra $KGL_\star KGL$ admits a decomposition mimicking the classical $KU$-cooperation algebra $KU_\ast KU$. We do not use this, but do use the simpler fact that $KGL \otimes KGL$ is a free $KGL$-module. It could be interesting to pursue an alternate approach which uses their theorem to give a more computational proof of \cref{MT:LX}. The authors hope that the reduction-to-$\c$ method pioneered here will prove useful in other applications where explicit computations may be inaccessible over general base schemes but feasible over $\c$.
\end{remark}

\subsection{Acknowledgements}

The authors thank Jeremiah Heller for explaining the extension of \cref{MT:LX} from fields to Dedekind domains, and thank Tom Bachmann for pointing out that cellularity is not necessary. They thank the anonymous referees for multiple helpful suggestions. The first author was supported by NSF RTG grant DMS-1839968; the second author was partially supported by NSF grant DMS-2204365; the third author was supported by NSF grants DMS-2039316 and DMS-2314082.

\section{First reductions}\label{Sec:FirstReduction}

By definition, to show that $L_{KGL/2}X \simeq (X\otimes JQ)_{(2,\eta)}^\wedge$, we must show that
\begin{enumerate}
\item $(X\otimes JQ)_{(2,\eta)}^\wedge$ is $KGL/2$-local, in the sense that if $C$ is $KGL/2$-acyclic then $\map(C,(X\otimes JQ)_{(2,\eta)}^\wedge) = 0$;
\item $X \to (X\otimes JQ)_{(2,\eta)}^\wedge$ is a $KGL/2$-equivalence.
\end{enumerate}

In this section, we prove (1), and show that (2) holds for all $X$ as soon as it holds for $X = \bone_S$. We will make frequent use of the following fact, due to R\"ondigs and \O{}stv\ae{}r.

\begin{lemma}[{\cite[Theorem 3.4]{RO16}}]\label{lem:kqmodeta}
There is an equivalence $KGL\simeq KQ/\eta$.
\qed
\end{lemma}

\begin{proposition}\label{Prop:local}
Let $X$ be an $S$-motivic spectrum. Then $(X\otimes JQ)_{(2,\eta)}^\wedge$ is $KGL/2$-local.
\end{proposition}
\begin{proof}
As the subcategory of $KGL/2$-local spectra is closed under fibers, it suffices to show that $(X \otimes KQ)_{(2,\eta)}^\wedge$ is $KGL/2$-local. Note that $KQ$ is an $\e_\infty$ ring spectrum \cite[Section 3.2]{BH20}. So to simplify the notation, let us show more generally that if $M$ is any $(2,\eta)$-complete $KQ$-module, then the underlying motivic spectrum of $M$ is $KGL/2$-local.

By \cref{lem:kqmodeta}, we may identify $KGL/2\simeq KQ/(2,\eta)$. It follows that if $C$ is $KGL/2$-acyclic, then $(KQ\otimes C)_{(2,\eta)}^\wedge\simeq 0$. From here we compute
\[
\map(C,M)\simeq\Mod_{KQ}(KQ\otimes C,M)\simeq\Mod_{KQ}((KQ\otimes C)_{(2,\eta)}^\wedge,M)\simeq 0,
\]
so that $M$ is $KGL/2$-local as claimed.
\end{proof}

\begin{proposition}\label{Prop:XtoS}
Suppose that $\bone_S \to JQ$ is a $KGL/2$-equivalence. Then $X \to (X\otimes JQ)_{(2,\eta)}^\wedge$ is a $KGL/2$-equivalence for any $S$-motivic spectrum $X$.
\end{proposition}
\begin{proof}
As $KGL/2\simeq KQ/(2,\eta)$, we have
\[
(X \otimes KQ)_{(2,\eta)}^\wedge \otimes KGL/2 \simeq X \otimes KQ \otimes KGL/2
\]
for any $X$. Thus $X \to (X\otimes JQ)_{(2,\eta)}^\wedge$ is a $KGL/2$-equivalence if and only if the natural map
\[
X \otimes KGL/2\rightarrow X \otimes JQ \otimes KGL/2
\]
is an equivalence. This map is obtained from the case $X = \bone_S$ by smashing with $X$, so if it is an equivalence for $X = \bone_S$ then it is an equivalence for any $X$.
\end{proof}

\section{Reduction to \texorpdfstring{$S = \Spec(\c)$}{S = Spec(C)}}\label{Sec:DtoC}

It remains to show that $\bone_S \to JQ$ is a $KGL/2$-equivalence. In this section, we show how one may reduce this statement over an arbitrary base scheme $S$ containing $1/2$ to just the case $S = \Spec(\c)$, which will then be carried out in \cref{Sec:C}.

\subsection{A conservativity result}

The heart of our argument is the following.

\begin{theorem}\label{thm:conservative}
Base change $\SH(\z[\tfrac{1}{2}])\rightarrow\SH(\c)$ is conservative when restricted to the full subcategory of $2$-torsion cellular motivic spectra with vanishing real Betti realization and convergent effective slice tower.
\end{theorem}

Here, we refer to the effective slice tower in the sense of Voevodsky \cite[Section 2]{Voe98}. The rest of this subsection is dedicated to the proof of \cref{thm:conservative}. In \cref{ssec:reduction}, we verify that this implies the desired reduction. The starting point for the proof is the following conservativity result of Bachmann and Hoyois.

\begin{lemma}\label{lem:conservativefield}
For $X\in\SH(\z[\tfrac{1}{2}])$ to vanish, it suffices that the base change of $X$ to $\Spec(k)$ vanishes for $k$ a prime field other than $\f_2$.
\end{lemma}
\begin{proof}
This is a special case of \cite[Proposition B.3]{BH21}.
\end{proof}

This would reduce us to considering fields, only if $k$ has positive characteristic then there is no obvious comparison between $\SH(k)$ and $\SH(\c)$. To handle this we make use of the following Lefschetz principle.

\begin{lemma}\label{lem:lefschetz}
Let $K$ be an algebraically closed field containing $1/2$. Then there is a zigzag of base change functors between $\SH(K)$ and $\SH(\c)$ which is an equivalence on full subcategories of $2$-torsion cellular obects.
\end{lemma}
\begin{proof}
This follows from \cite[Proposition 5.2.1]{BCQ21} and its proof upon restricting to $2$-torsion objects.
\end{proof}

Now fix a prime field $k$ other than $\f_2$, and fix an algebraic closure $p\colon \Spec(K)\rightarrow\Spec(k)$. The core of our argument is to show that $p^\ast\colon \SH(k)\rightarrow\SH(K)$ is conservative when restricted to the full subcategory of $2$-torsion cellular motivic spectra with convergent slice tower, and with vanishing real Betti realization when $k = \q$. As $p^\ast$ is exact, it is equivalent to show that if $X\in\SH(k)$ is such an object and $p^\ast X = 0$, then $X = 0$. First let us reformulate the condition on the Betti realization of $X$ when $k = \q$ to one which is uniform in $k$. Recall that $\rho = [-1] \in \pi_{-1,-1}\bone_k$ (see \cref{ssec:conventions}), and write $C(\rho)$ for its cofiber.

\begin{lemma}\label{lem:realbettirho}
Let $X\in\SH(k)$, and when $k = \q$ suppose that the real Betti realization of $X$ vanishes. If $C(\rho)\otimes X  = 0$ then $X = 0$.
\end{lemma}
\begin{proof}
We claim that $X$ is $\rho$-torsion, that is that $X[\rho^{-1}] = 0$. The claim follows as smashing with $C(\rho)$ is conservative on $\rho$-torsion objects, as can be seen using the cofiber sequence $\bone_S\rightarrow\bone_S[\rho^{-1}]\rightarrow \bone_S/(\rho^\infty)$, where $\bone_S/(\rho^\infty) = \colim_n \Sigma^{n,n}C(\rho^n)$ is built out of copies of $C(\rho)$.

First consider $k = \bbF_p$ ($p\neq 2$). Here $\rho$ is already nilpotent in $\pi_{\ast,\ast}\bone_S$ \cite[Example 1.5]{Mil70}, and so $X[\rho^{-1}]\simeq 0$ for any $X$.

Next consider $k = \q$. As $\q$ has a unique ordering, work of Bachmann \cite[Theorem 35, Proposition 36]{Bac18} shows that $\SH(\q)[\rho^{-1}]\simeq\SH$ via real Betti realization. Thus the condition that the real Betti realization of $X$ vanishes exactly says that $X[\rho^{-1}] = 0$.
\end{proof}

We now turn to considering slices. For a motivic spectrum $X$, write $f_q X$ for the $q$-effective cover of $X$, and write $s_q X = f_qX/f_{q+1}X$ for the $q$th slice of $X$. 

\begin{lemma}\label{lem:sliceeilen}
Let $k$ be any field containing $1/2$. Then the assignment $X\mapsto s_\ast(X/2)$ lifts to an exact functor $\SH(k)\rightarrow\smash{\Mod_{H\f_2^k}}$. Moreover, if $X\in\SH(k)$ is cellular then $s_\ast(X/2)$ is cellular as an $H\f_2^k$-module.
\end{lemma}
\begin{proof}
Exactness follows from the definitions, see \cite[Theorem 2.2]{Voe98}. The assignment $X \mapsto s_\ast(X)$ defines a lax monoidal functor from $\SH(k)$ to $\z$-graded objects in $\SH(k)$, see \cite{GRSO12}. In particular, $s_\ast(X)$ is a module over $s_0(\bone_k)$. Theorems of Voevodsky \cite[Thm. 6.6]{Voe04} and Levine \cite[Theorem 10.5.1]{Lev08} identify $s_0(\bone_k)\simeq H\bbZ^k$, and so exactness of $s_\ast$ implies that
\[
s_\ast(X/2)\simeq s_\ast(X)/2\simeq H\bbF_2^k\otimes_{H\bbZ^k}s_\ast(X)
\]
is naturally an $H\bbF_2^k$-module as claimed. Finally, to verify that if $X$ is cellular then $s_\ast(X/2)$ is cellular as an $H\f_2^k$-module, it suffices to verify that $s_\ast(\bone_k)_{(2)}$ is cellular as an $H\z_{(2)}^k$-module. As $2$ is invertible in $k$, this follows from R\"ondigs--Spitzweck--\O{}stv\ae{}r's computation of the slices $s_\ast(\bone_k)_{(2)}$ in \cite[Theorem 2.12]{RSO19}, applying their theorem in the case $S = \Spec(k)$ and $\Lambda = A = \bbZ_{(2)}$.
\end{proof}

Our last main ingredient in the proof of \cref{thm:conservative} is to show that $\Mod_{H\f_2^k}\rightarrow\Mod_{H\f_2^K}$ is conservative upon restriction to cellular $\rho$-torsion objects. The proof of this is a variant of \cite[Example 3.5]{Mat18}. Write $\Cell$ for the cellularization functor, see \cite[Section 4]{behrensshah2020c2}. Note that $\Cell$ does not affect homotopy groups, and set
\[
E = \Cell(\Spec(K)_+\otimes H\f_2^k) \in \Mod_{H\f_2^k}.
\]
Recall that $\pi_{\ast,\ast}H\f_2^k = k_\ast^M(k)[\tau]$, where $\tau \in \pi_{0,-1} H\f_2^k$ and $x\in k_n^M(k)$ lives in $\pi_{-n,-n}H\f_2^k$, and likewise $\pi_{\ast,\ast}E = \f_2[\tau]$. Say that an $H\f_2^k$-module $M$ is \textit{$\tau$-free} if $\pi_{\ast,\ast}M$ is free as an $\f_2[\tau]$-module.

\begin{lemma}\label{lem:cotruncation}
Let $M$ be a $\tau$-free $H\bbF_2^k$-module. Then there exists a $\tau$-free $H\bbF_2^k$-module $M_{\geq n}$ receiving a map $M\rightarrow M_{\geq n}$ with the following properties:
\begin{enumerate}
\item $\pi_{i,\ast}M_{\geq n} = 0$ for $i < n$;
\item $\pi_{i,\ast}M\rightarrow \pi_{i,\ast}M_{\geq n}$ is an isomorphism for $i\geq n$.
\end{enumerate}
\end{lemma}
\begin{proof}
The proof of \cite[Proposition 3.3]{DGI06} applies. For $m\in\bbZ$, choose a basis $\{x_i : i\in I_m\}$ for $\pi_{m,\ast}M$ as an $\bbF_2[\tau]$-module, and define
\[
V_m M = \bigoplus_{x\in I_m}\Sigma^{|x_i|}H\bbF_2^k.
\]
Then there is a map $V_m M \rightarrow M$ inducing an isomorphism on $\pi_{m,\ast}$. Let
\[
C(M) = \cof\left(\bigoplus_{m<n}V_m(M)\rightarrow M\right).
\]
Coconnectivity of $H\bbF_2^k$ ensures that $M\rightarrow C(M)$ induces an isomorphism on $\pi_{i,\ast}$ for $i\geq n$ and is zero on $\pi_{i,\ast}$ for $i < n$. In particular, $C(M)$ is $\tau$-free: if $i\geq n$ then $\pi_{i,\ast}C(M)\cong \pi_{i,\ast}M$, so this follows from $\tau$-freeness of $M$, and if $i<n$ then $\pi_{i,\ast}C(M)\subset \pi_{i-1,\ast}\bigoplus_{m<n}V_m(M)$, so this follows from $\tau$-freeness of each $V_m(M)$. Thus iterating this construction and setting $M_{\geq n} = \colim_k C^k(M)$ does the job.
\end{proof}

\begin{lemma}\label{lem:gr}
Let $M$ be a $\tau$-free cellular $H\bbF_2^k$-module. Fix $n\in\bbZ$ and suppose $\pi_{i,\ast}M = 0$ for $i\neq n$. Then $M$ is equivalent to a sum of copies of $\Sigma^{n,\ast}E$.
\end{lemma}
\begin{proof}
Without loss of generality we may suppose $n = 0$. First we claim that if $x\in \pi_{0,w}M$, then there is a map $g\colon \Sigma^{0,w}E\rightarrow M$ satisfying $g(1) = x$. Without loss of generality we may suppose $w = 0$. Note that $\pi_{\ast,\ast}E\cong \f_2[\tau]$, and consider the universal coefficient spectral sequence
\[
E_2^{s,f,w} = \Ext_{\pi_{\ast,\ast}H\bbF_2^k}^f(\bbF_2[\tau],\pi_{\ast-s,\ast-w}M) \Rightarrow [E,\Sigma^{s+f,w}M],\quad d_r\colon E_r^{s,f,w}\rightarrow E_r^{s+1,f+r,w}.
\]
The proposed map $g$ defines a class in $E_2^{0,0,0}$. The only way this could fail to survive the spectral sequence is if there is some nontrivial differential $d_r(g)\neq 0$. This lives in a subquotient of $\Ext^r_{\pi_{\ast,\ast}H\bbF_2^k}(\bbF_2[\tau],\pi_{\ast-1,\ast}M)$, which vanishes because $M$ is concentrated in nonnegative stems and $\pi_{\ast,\ast}H\bbF_2^k$ and $\bbF_2[\tau]$ are concentrated in nonpositive stems. Thus $g$ survives to a map $g\colon \Sigma^{0,w}E\rightarrow M$ satisfying $g(1) = x$ as claimed.

Now choose a basis $\{x_i:i\in I\}$ for $\pi_{0,\ast}M$ as an $\bbF_2[\tau]$-module. Then the above argument provides a map
\[
(x_i)_{i\in I}\colon \bigoplus_{i\in I}E \rightarrow M
\]
inducing an isomorphism in $\pi_{\ast,\ast}$, which is then an equivalence as $E$ and $M$ are cellular.
\end{proof}

The above lemmas hold in general over any base field $k$ containing $1/2$. However, the following relies on $k$ being a prime field other than $\f_2$ in order to ensure that $\pi_{\ast,\ast}C(\rho)\otimes H\f_2^k$ is concentrated in finitely many degrees.

\begin{proposition}\label{prop:conservative}
Base change $p^\ast\colon \Mod_{H\bbF_2^k}\rightarrow\Mod_{H\bbF_2^K}$ is conservative when restricted to the full subcategory of cellular $\rho$-torsion $H\f_2^k$-modules.
\end{proposition}
\begin{proof}
Write $p\colon \Spec(K)\rightarrow\Spec(k)$. By the definition of $E$, the composite
\begin{center}\begin{tikzcd}
\Mod_{H\bbF_2^k}^{\cell}\ar[r,"\subset"]&\Mod_{H\bbF_2^k}\ar[r,"p^\ast"]&\Mod_{H\bbF_2^K}\ar[r,"p_\ast"]&\Mod_{H\bbF_2^k}\ar[r,"\Cell"]&\Mod_{H\bbF_2^k}^{\cell}
\end{tikzcd}\end{center}
sends $\Sigma^{i,j}H\bbF_2^k$ to $\Sigma^{i,j}E$. As $\Mod_{H\bbF_2^k}^{\cell}$ is generated under colimits by the modules $\Sigma^{i,j}H\bbF_2^k$, and each functor in this composition preserves colimits, this identifies the composite as smashing with $E$. It therefore suffices to show that smashing with $E$ is conservative on the full subcategory of $\rho$-torsion $H\bbF_2^k$-modules.

Examination of the long exact sequence shows that $C(\rho)$ is $\tau$-free for any field $k$. When $k$ is a prime field other than $\f_2$, the kernel and cokernel of $\rho$ acting on $k_\ast^M(k)$ is concentrated in degrees $0\leq \ast \leq 2$ (convenient references are \cite[Section 2.1]{IO18} and \cite[Section 5]{OO13}), and further examination of the long exact sequence then shows that $\pi_{i,\ast}C(\rho) \neq 0$ only for $-3\leq i \leq 0$. Combining \cref{lem:cotruncation} and \cref{lem:gr}, we find that $C(\rho)$ admits a finite filtration with associated graded equivalent to a direct sum of copies of $\Sigma^{\ast,\ast}E$. In particular $C(\rho)$ lies in the thick $\otimes$-ideal of $\Mod_{H\bbF_2^k}$ generated by $E$.

Smashing with $C(\rho)$ is conservative on the full subcategory of $\rho$-torsion $H\f_2^k$-modules. As $C(\rho)$ lies in the thick $\otimes$-ideal generated by $E$, it follows that smashing with $E$ is conservative on the full subcategory of $\rho$-torsion $H\f_2^k$-modules, proving the proposition.
\end{proof}

We can now give the following.

\begin{proof}[Proof of \cref{thm:conservative}]
Combining \cref{lem:conservativefield}, \cref{lem:lefschetz}, and \cref{lem:realbettirho}, we reduce to verifying the following assertion:

\begin{quote}
Let $k$ be a prime field other than $\f_2$ and $p\colon \Spec(K)\rightarrow\Spec(k)$ be an algebraic closure. Let $X\in\SH(k)$ be a cellular motivic spectrum with convergent slice tower, and suppose $p^\ast(C(\rho)\otimes X/2) = 0$. Then $C(\rho)\otimes X/2 = 0$.
\end{quote}

Indeed, as $X$ is slice complete, so is $C(\rho)\otimes X/2$. It therefore suffices to verify that $s_\ast(C(\rho)\otimes X/2) = 0$. By exactness $s_\ast(C(\rho)\otimes X/2)\simeq C(\rho)\otimes s_\ast(X/2)$, and so by \cref{lem:sliceeilen} and \cref{prop:conservative} it suffices to verify that $p^\ast(s_\ast(C(\rho)\otimes X/2)) = 0$. As slices of cellular spectra are preserved by base change \cite[Corollary 2.17]{RSO19}, we have $p^\ast(s_\ast(C(\rho)\otimes X/2))\simeq s_\ast(p^\ast(C(\rho)\otimes X/2))$, and this vanishes by assumption.
\end{proof}

\subsection{Proof of the reduction to \texorpdfstring{$S = \Spec(\c)$}{S = Spec(C)}}\label{ssec:reduction}

We now show that \cref{thm:conservative} allows us to reduce from verifying that $\bone_S\rightarrow JQ$ is a $KGL/2$-equivalence to just the case $S = \Spec(\c)$.

\begin{lemma}\label{lem:basechange}
The following $S$-motivic spectra are cellular and compatible with base change between schemes containing $1/2$:
\[
KQ,\qquad KGL,\qquad KGL/2,\qquad KGL/2 \otimes JQ.
\]
\end{lemma}
\begin{proof}
Recall that we have taken $KQ$ and $KGL$ to be defined as in \cite[\S 3.2]{BH20}. They are cellular and compatible with base change as discussed there. As $KGL/2\simeq KQ/(2,\eta)$, we may identify 
\[
KGL/2\otimes JQ\simeq KGL/2\otimes \fib\left(\psi^3-1\colon KQ[\tfrac{1}{3}] \rightarrow KQ[\tfrac{1}{3}]\right).
\]
This is then cellular, and is compatible with base change between schemes containing $1/2$ by \cite[Theorem 3.1(1)]{BH20}.
\end{proof}

\begin{lemma}\label{lem:slicecomplete}
The motivic spectra $KGL/2$ and $KGL/2\otimes JQ$ are slice complete.
\end{lemma}
\begin{proof}
As the functors $f_q$ are exact, slice complete motivic spectra form a thick subcategory of $\SH(S)$. As $KGL/2\simeq KQ/(2,\eta)$, there is an equivalence $KGL/2 \otimes (KQ)_{(2,\eta)}^\wedge\simeq KGL/2\otimes KQ$; moreover, $\eta$ acts trivially on $KGL\otimes KQ$, and thus $KGL\otimes KQ$ is a retract of $(KGL\otimes KQ)/\eta\simeq KGL\otimes KGL$. Thus we reduce to verifying that $KGL$ and $KGL\otimes KGL$ are slice complete. Following the proof of \cite[Lemma 2.6]{ARO17}, $f_q KGL$ is $q$-connected in the sense of \cite[Definition 3.16]{RSO19}. As $\infty$-connected objects are contractible, it follows that $KGL$ is slice complete. By \cite[Corollary 7.5]{pelaezweibel2014slices}, $KGL\otimes KGL$ is a free $KGL$-module. As $f_q$ commutes with direct sums \cite[Proposition 6.1]{pelaezweibel2014slices}, it follows that $f_q(KGL\otimes KGL)$ is also $q$-connected, and thus $KGL\otimes KGL$ is slice complete.
\end{proof}

\begin{lemma}\label{lem:kglbetti}
The real Betti realizations of $KGL/2$ and $KGL/2 \otimes JQ$ vanish.
\end{lemma}
\begin{proof}
As real Betti realization is exact and monoidal, it suffices to verify that the real Betti realization of $KGL$ vanishes. This is proved by Bachmann and Hopkins in \cite[Lemma 3.9]{BH20}.
\end{proof}

As far as the proof of \cref{MT:LX} is concerned, the following is the main result of this section.

\begin{proposition}\label{prop:creduce}
To show that $\bone_S\rightarrow JQ$ is a $KGL/2$-equivalence over any scheme $S$ containing $1/2$, it suffices to consider just $S = \Spec(\c)$.
\end{proposition}
\begin{proof}
\cref{lem:basechange} reduces the case of an arbitrary scheme $S$ containing $1/2$ to just $S = \Spec(\z[\tfrac{1}{2}])$. By \cref{lem:slicecomplete} and \cref{lem:kglbetti}, the hypotheses of \cref{thm:conservative} apply to then reduce to $S = \Spec(\c)$.
\end{proof}

\section{The case \texorpdfstring{$S=\Spec(\c)$}{S=Spec(C)}}\label{Sec:C}

It remains to show that $\bone_\c\rightarrow JQ$ is a $KGL/2$-equivalence. To do this, we will make use of the close relationship between $\c$-motivic homotopy theory and the classical Adams--Novikov spectral sequence. There are (at least) two closely related approaches to formalizing this relation: the synthetic spectra of Pstragowski \cite{Pst23}, and the $\Gamma_\star$-construction of Gheorghe--Isaksen--Krause--Ricka \cite{GIKR22}. For our argument it is convenient to use the latter. 

In \cite{GIKR22}, properties of the classical and motivic Adams--Novikov spectral sequences are used to construct:
\begin{enumerate}
\item A lax symmetric monoidal functor $\Gamma_\star\colon \Sp\rightarrow\Fun((\z,<)^\op,\Sp)$, and 
\item A lax symmetric monoidal functor $\Omega^{0,\star}\colon \SH(\c)\rightarrow\Mod_{\Gamma_\star(S)}$,
\end{enumerate}
and to prove that $\Omega^{0,\star}$ induces an equivalence between the category of $2$-complete cellular $\c$-motivic spectra and the category of $2$-complete $\Gamma_\star(S)$-modules. Moreover, under this correspondence one has
\begin{equation}\label{eq:dictionary}
KGL\simeq \Gamma_\star(KU),\qquad KQ\simeq \Gamma_\star(KO),
\end{equation}
up to $2$-completion. We are now ready to prove the following.

\begin{proposition}\label{prop:cequiv}
The map $\bone_\c\rightarrow JQ$ is a $KGL/2$-equivalence.
\end{proposition}
\begin{proof}
It is equivalent to show that
\begin{equation}\label{eq:fibmot}
\begin{tikzcd}
KGL/2 \ar[r]&KGL/2 \otimes KQ \ar[r,"\psi^3-1"]&KGL/2 \otimes KQ
\end{tikzcd}\end{equation}
is a fiber sequence of $\c$-motivic spectra. Consider the classical fiber sequence
\begin{center}\begin{tikzcd}
KU/2\ar[r]&KU/2\otimes KO\ar[r,"\psi^3-1"]&KU/2 \otimes KO
\end{tikzcd}.\end{center}
As these spectra have $MU$-homology concentrated in even degrees, \cite[Proposition 3.17]{GIKR22} implies that $\Gamma_\star$ sends this to a fiber sequence
\begin{equation}\label{eq:fibsynth}\begin{tikzcd}
\Gamma_\star(KU/2)\ar[r]&\Gamma_\star(KU/2 \otimes KO)\ar[r,"\psi^3-1"]&\Gamma_\star(KU/2 \otimes KO)
\end{tikzcd}\end{equation}
of $\Gamma_\star(S)$-modules. We claim that $\Gamma_\star(KU/2)\simeq \Gamma_\star(KU)/2$ and $\Gamma_\star(KU/2\otimes KO)\simeq \Gamma_\star(KU/2)\otimes_{\Gamma_\star(S)}\Gamma_\star(KO)$. In light of the equivalence between $2$-complete cellular motivic spectra and $2$-complete $\Gamma_\star(S)$-modules, this would give an equivalence between the two sequences \cref{eq:fibmot} and \cref{eq:fibsynth}, and thus the former would be a fiber sequence as claimed.

That $\Gamma_\star(KU/2)\simeq\Gamma_\star(KU)/2$ follows by again applying \cite[Proposition 3.17]{GIKR22}, this time to the cofiber sequence $KU\rightarrow KU \rightarrow KU/2$; likewise one has $\Gamma_\star(KU/2 \otimes KO)\simeq \Gamma_\star(KU\otimes KO)/2$. Recall that $KU$ is a filtered colimit of finite complexes with only even cells \cite[Proposition 2.12]{hoveystrickland1999morava}. As $KU\simeq KO/(\eta)$ and $\eta$ acts trivially on $MU$, it follows that $KO$ has $MU$-homology concentrated in even degrees. Thus \cite[Proposition 3.25]{GIKR22} implies that $\Gamma_\star(KU \otimes KO)\simeq \Gamma_\star(KU)\otimes_{\Gamma_\star(S)}\Gamma_\star(KO)$. Combining these observations provides the requires equivalences.
\end{proof}

Combining \cref{Prop:local}, \cref{Prop:XtoS}, \cref{prop:creduce}, and \cref{prop:cequiv} now proves \cref{MT:LX}.

\bibliographystyle{alpha}
\bibliography{master}

\end{document}